\documentclass{amsart}
\usepackage{amssymb, amsmath, amscd, latexsym, mathrsfs}
\usepackage{tikz} 
\usetikzlibrary{matrix,arrows} 

\usepackage{xr-hyper}

\usepackage[pagebackref=true,pdfnewwindow=true,pdftex]{hyperref}
\hypersetup{colorlinks, linkcolor=blue, citecolor=red}

\usepackage{macros}

\usepackage{verbatim} % for block comments

\begin{document}

%% citations
%\nocite{*}

\title{Manifold calculus and homotopy sheaves}%
\author{Pedro Boavida de Brito and Michael Weiss}%
\address{Mathematisches Institut, Universit\"{a}t M\"{u}nster, Einsteinstr. 62, D-48149 M\"{u}nster}
%\address{Dept. of Mathematics, University of Aberdeen, AB24 3UE, United Kingdom}%
\email{p.boavida@uni-muenster.de}
\email{m.weiss@uni-muenster.de}
%\subjclass{}
%\keywords{}%
%\date{\today}%
\thanks{The first author was supported by FCT, Funda\c{c}\~{a}o para a Ci\^{e}ncia e a Tecnologia, grant SFRH/BD/61499/2009.\\
The second author was supported through a Humboldt professorship.}
%\dedicatory{tsdasdas}%
%\commby{}
% ----------------------------------------------------------------

\begin{abstract}
Manifold calculus is a form of functor calculus concerned with contravariant functors from some category of manifolds to spaces. A weakness in the original formulation is that it is not continuous in the sense that it does not handle the natural enrichments well. In this paper, we correct this by defining an enriched version of manifold calculus which essentially extends the discrete setting. Along the way, we recast the Taylor tower as a tower of homotopy sheafifications. As a spin-off we obtain a natural connection to operads: the limit of the Taylor tower is a certain (derived) space of right module maps over the framed little discs operad.
\end{abstract}
\maketitle
% ----------------------------------------------------------------

%%%%%%%%%%%%%%%%%%%%%%%%%%%%%%%%%%%%%%%% %%%%%% %%%%
%%%%%%%%%%%%%%%%%%%%%%%%%%%%%%%%%%%%%%%% SECTION 0 %%%%
%%%%%%%%%%%%%%%%%%%%%%%%%%%%%%%%%%%%%%%% %%%%%% %%%%

\section{Introduction}

\quad Let $M$ be a smooth manifold without boundary and denote by $\sO(M)$ the poset of open subsets of $M$, ordered by inclusion. Manifold calculus, as defined in \cite{embI}, is a way to study (say, the homotopy type of) contravariant functors $F$ from $\sO(M)$ to spaces which take isotopy equivalences to (weak) homotopy equivalences. In essence, it associates to such a functor a tower - called the \emph{Taylor tower} - of polynomial approximations which in good cases converges to the original functor, very much like the approximation of a function by its Taylor series. \\ %For a survey and history, consult \cite{GKW}. \\

\quad The remarkable fact, which is where the geometry of manifolds comes in, is that the Taylor tower can be explicitly constructed: the $k^{th}$ Taylor polynomial of a functor $F$ is a functor $T_k F$ which is in some sense the universal approximation to $F$ with respect to the subposet of $\sO(M)$ consisting of open sets diffeomorphic to $k$ or fewer open balls.\\ 

\quad A weakness in the traditional discrete approach is that in cases where $F$ has obvious continuity properties, $T_k F$ does not obviously inherit them, where by continuous we mean enriched over spaces. For example, let $F(U)$ for $U\in \sO(M)$ be the space of smooth embeddings from $U$ to a fixed smooth manifold $N$. It is clear that the group of diffeomorphisms $M\to M$ acts in a continuous manner on $F(M)$. One would expect a similar continuous action of the same group on $T_kF(M)$, for all $k$. But with the standard description of $T_kF$ we only get an action of the underlying discrete group. As a solution to this problem in the particular case of the embedding functor a continuous model, Haefliger-style, was proposed in \cite{GKW}.  \\

\quad In this paper, we correct this lack of continuity by defining an enriched (or $\infty$) version of manifold calculus. Along the way, we reapproach the foundations of the theory by focusing on the wider notion of homotopy sheaves rather than on polynomial functors which had the central role in \cite{embI}. \\

\quad We now give a brief overview of the paper. Let $\sS$ be a category of \emph{spaces}, i.e. compactly generated Hausdorff spaces or simplicial sets (more on this at the end of the introduction). To have an enriched setting we replace the category $\sO(M)$ by the topological category $\Man$ of smooth manifolds of a fixed dimension $d$ and (codimension zero) embeddings. We then want to consider contravariant functors which are enriched over $\sS$, namely functors $F : \Man^{op} \rightarrow \sS$ inducing \emph{continuous} (or simplicial) maps
\[
\emb(M,N) \longrightarrow \Hom_{\sS}(F(N), F(M))
\]
which preserve the units and composition. 

\quad Moreover, there is the usual Grothendieck topology $\sJ_1$ on $\Man$ given by open covers. For each positive $k$, we can define a multi-local version of $\sJ_1$ where we only admit covers which have the property that every set of $k$ (or fewer) points is contained in some open set of the cover. These form the Grothendieck topologies $\sJ_k$. % (see Definition \ref{Jk}). 
Since $\Man$ is now viewed as a site, we refer to $\sS$-functors on $\Man$ as $\sS$-enriched presheaves, or simply \emph{presheaves}.

\begin{defn}
The \textbf{Taylor tower} of $F$ is the \emph{tower of homotopy sheafifications} of $F$ with respect to the Grothendieck topologies $\sJ_k$.
\end{defn}

\quad For the precise meaning of this statement, see section \ref{section1b}. The enriched analogue of $T_k F$ is an $\sS$-enriched presheaf denoted by $\sT_k F$ (Definition \ref{TkF}). It is the best homotopical approximation to $F$ with respect to the subcategories $\Dk$ of $\Man$ whose objects are disjoint unions of $k$ or fewer balls, and it comes with a natural `evaluation' map 
\begin{equation}\label{TA}
F \rightarrow \sT_k F
\end{equation}

One of the main results of this paper is
\begin{thm} 
The map (\ref{TA}) is a homotopy $\sJ_k$-sheafification.
\end{thm}

As a byproduct, we obtain a very natural connection to operads,
\begin{equation*}\label{mods}
\sT_{\infty}ÊF (M) \simeq \mathbb{R}\Hom_{P}(\emb_M, F)
\end{equation*}
where the right hand side is the derived space of right module maps over the framed little \mbox{$d$-dim} discs operad $P$, and $\emb_M$ and $F$ are the right $P$-modules defined by $\{\emb(\amalg_n \RR^d, M)\}_{n \geq 0}$ and $\{F(\amalg_n \RR^d)\}_{n \geq 0}$, respectively.  This answers a conjecture of Greg Arone and Victor Turchin (Conjecture 4.14, \cite{AroneTurchin}).

\quad In the case where $F$ is the embedding functor $\textup{Emb}(-,N)$ and $\textup{dim} N - d \geq 3$ we get, as an immediate corollary of Goodwillie-Klein excision estimates, that
\[
\emb(M,N) \simeq \mathbb{R}\Hom_{P}(\emb_M, \emb_N)
\]

\quad A discrete version of this connection to operads appeared recently in the work of Arone and Turchin [ibid.] where, coupled with formality results, it is further used to obtain explicit descriptions of the rational homology and homotopy of certain spaces of embeddings.\\

\quad Finally, we point out that the framework in this paper is rather general and can be applied to other categories other than $\Man$. Namely, for a topological (or $\infty$) category $\sC$ equipped with a Grothendieck topology possessing \emph{good covers} and, given a presheaf $F$ on $\sC$, one can construct the tower of homotopy sheafifications of $F$ - its Taylor Tower - and give an explicit model for it as a tower of homotopical approximations with respect to certain subcategories of $\sC$. Examples include the category of topological spaces, the category of $d$-dimensional manifolds with boundary, the category of all manifolds, and the analogous versions where instead of smooth manifolds one considers topological manifolds.

\subsection*{Outline of the paper}
In section 2 we define homotopy sheaves. We relax the definition of a Grothendieck topology to that of a coverage and we introduce two coverages $\sJ_k^{\circ}$ and $\sJ_k^{h}$. We show in section 5 and 7, respectively, that $\sJ_k^{\circ}$ and $\sJ_k^{h}$ form a basis for the Grothendieck topology $\sJ_k$ by proving that the three coverages generate the same homotopy sheaves. To set the ground, we first introduce the local model structure on the category of presheaves in section 3 and, in section 4, we discuss enriched homotopical (or $\infty$) Kan extensions. Finally, in section 8 we show that $\sT_k$ is really an `enrichment' of $T_k$. Specifically, we show that for functors $F$ on $\sO(M)$ which, like $\emb(-,N)$, factor through 
$\Man$, we have a weak equivalence
\[ T_k F (U) \simeq \sT_k F (U) \]
for every open set $U$ of $M$.

\subsection*{Spaces, enrichments and notation. }
We do not want to be very imposing on which category of spaces we work with. However, we need it to be cartesian closed, considered as enriched over itself and having small limits and colimits. The category of compactly generated Hausdorff spaces is a natural candidate and the one we opt for, but everything can easily be formulated simplicially (see Appendix). We denote this category of spaces by $\sS$.
To make $\sS$ enriched over itself give the $\Hom$-sets, $\Hom_{\sS}(X,Y)$, the (Kelleyfication of the) weak\footnote{The same would not be true if we had taken, for instance, the strong topology.} topology  (compact-open topology).

\quad Similarly, the category $\Man$ of $d$-dimensional manifolds without boundary is enriched over $\sS$: give the $C^{\infty}$ weak topology to the space of smooth embeddings $\emb(M,N)$. One important property of the weak topology is that it is metrizable, hence compactly generated and Hausdorff.

\quad All manifolds in this paper are assumed to be paracompact and, except in section 9, \textbf{without boundary}.

\subsection*{Acknowledgements}
We would like to thank Victor Turchin for numerous helpful comments and corrections. We also thank Assaf Libman and Mike Shulman for interesting discussions on Kan extensions and higher categories, and Dimitri Zaganidis for a correction. Last, but not least, it is a pleasure to acknowledge the influence of the writings of Dan Dugger on this paper.

%%%%%%%%%%%%%%%%%%%%%%%%%%%%%%%%%%%%%%%% %%%%%% %%%%
%%%%%%%%%%%%%%%%%%%%%%%%%%%%%%%%%%%%%%%% SECTION 1 %%%%
%%%%%%%%%%%%%%%%%%%%%%%%%%%%%%%%%%%%%%%% %%%%%% %%%%

\section{Homotopy sheaves}\label{section1}

\begin{defn}Let $\sC$ be a (small) category. A \emph{coverage} $\tau$ is an assignment to each object $X \in \sC$ of a set $Cov_{\tau}(X)$ of collections of objects in the overcategory $\sC \downarrow X$ subject to the following condition: 
Given $\sU := \{U_i \rightarrow X\}_{i \in I}$ in $Cov_{\tau}(X)$ and a finite subset $S:=\{i_0, \dots, i_n\}$ of $I$, the iterated pullback $U_S := U_{i_0} \times_X \dots \times_X U_{i_n}$ exists in $\sC$.

\quad An element $\sU \in Cov_{\tau}(X)$ is called a \emph{covering} of $X$.
\end{defn}

\quad Let $(\sC, \tau)$ be a simplicial or topological (i.e. $\sS$-enriched) category equipped with a coverage $\tau$. We denote by $\PSh(\sC)$ the category of simplicial or topological presheaves on $\sC$ (i.e. $\sS$-enriched functors $\sC^{op} \rightarrow \sS$). Since we will mostly be dealing with $\sS$-enriched objects, we will often drop the adjective `enriched'.

\begin{defn}\label{descent} A presheaf $F \in \PSh(\sC)$ is said to satisfy \emph{descent} for a covering $\sU := \{U_i \rightarrow X\}_{i \in I}$ in $\tau$ if the natural map
\[
F(X) \longrightarrow \holimsub{S \subset I} F(U_S)
\]
is a weak equivalence. The homotopy limit ranges over all non-empty, finite subsets $S$ of $I$ and, for $S= \{i_0, \dots, i_n\}$, the object $U_S$ is the iterated pullback  $U_{i_0} \times_X \dots \times_X U_{i_n}$.

\quad A presheaf $F$ is a \emph{homotopy $\tau$-sheaf} (or satisfies $\tau$-descent) if it satisfies descent for every covering in $\tau$.
\end{defn}

\begin{rem} 
A presheaf $F$ is said to satisfy \v{C}ech descent for a covering $\sU := \{U_i \rightarrow X\}_{i \in I}$ if the natural map 
\[
F(X) \longrightarrow \holimsub{[n] \in \Delta} \prod_{i_0, \dots, i_n} F(U_{\{i_0, \dots, i_n\}})
\]
is a weak equivalence. This definition is equivalent to our definition \ref{descent} by cofinality (\cite{BK}, p317).
\end{rem}

\subsection{Coverages on the category of manifolds} 
Fix $d \geq 0$ once and for all. Let $\Man$ be the category of $d$-dimensional smooth manifolds and codimension zero embeddings. To ensure we have a small category, we consider its objects to be $d$-dimensional smooth submanifolds of $\RR^{\infty}$. 

\quad Since $\Man$ has pullbacks (which are given by intersection), the condition defining a coverage on this category is vacuous. Manifold calculus provides us with two standard examples of coverages.

\begin{defn}[Coverage $\sJ_k$]  The coverings of $M$ in $\sJ_k$ are given by the collection of morphisms in $\Man$ of the form
\[
\{ f_i : U_i \rightarrow M \}_{i \in I}
\]
such that every set of $k$ or fewer points is contained in $U_i$ for some $i \in I$. These are called $k$-\emph{covers}.
\end{defn}

Clearly, a $1$-cover is the usual notion of an open cover of a manifold.

\begin{defn}[Coverage $\sJ_k^{h}$] The coverings of $M$ in $\sJ_k^h$ are given by the collection of morphisms in $\Man$ of the form
\[
\{ f_i : M \backslash A_i \rightarrow M \}_{i \in \{0,\dots, k\}}
\]
where $A_0, \dots, A_k$ are disjoint closed subsets of $M$.
\end{defn}

\quad Furthermore, we declare the set of coverings of the empty set $Cov(\varnothing)$ for all our coverages consists of a single element $\{ id: \varnothing \rightarrow \varnothing \}$. Notice that if we included the empty collection of morphisms in the set of coverings of the empty set, then our sheaves would have the property that $F(\varnothing)$ is contractible.

\begin{rem} A homotopy sheaf for $\sJ_k^{h}$ is usually called a \emph{polynomial functor} of degree $\leq k$.  
\end{rem}

\subsection{Generalised good covers}
\begin{defn}\label{Ek}
Define the full subcategory $\Dk$ of $\Man$ by
\[
Ob(\Dk) := \Bigl\{ \mbox{manifolds diffeomorphic to $\underset{j \leq k}{\coprod} \RR^d$ for some $0 \leq j \leq k$Ê} \Bigr\}
\]
\end{defn}
\begin{rem}
The empty set $\varnothing$ is also an object of $\Dk$ (this is, by convention, the case $j = 0$).
\end{rem}

\quad It was realised long ago that every manifold $M$ admits a covering $\{U_i \rightarrow M\}$ such that all finite intersections belong to $\D_1$ (see, for instance, \cite{BottTu}, Theorem 5.1). In other words, every manifold can be covered by open balls $\{U_i\}$ such that every finite non-empty intersection $U_{i_0} \cap \dots \cap U_{i_n}$ is again diffeomorphic to an open ball. These are usually called \textit{good covers}. Good covers clearly define a coverage $\sJ_1^{\circ}$ on $\Man$.

\begin{defn}\label{goodkcover} A cover $\{U_i \rightarrow M\}_{i \in I}$ of a manifold $M$ is called a  \emph{good $k$-cover} if 
\begin{enumerate}
\item every set of $k$ or fewer points is contained in $U_i$ for some $i$ in $I$
\item every finite intersection  $U_{i_0} \cap \dots \cap U_{i_n}$ belongs to $\Dk$
\end{enumerate}
\end{defn}
\quad A good  $1$-cover is simply a good cover. The multi-local analogue of the paragraph above is

%%%%%%%%%%%%%%%%%%%%%%%%%%%%%%%%%%%%

\begin{prop}\label{admits_good}
Every manifold $M$ admits a good $k$-cover.
\end{prop}
\begin{proof}

Equip $M$ with a Riemannian metric, which we may take to be complete. Then there is, between any two points $x$ and $y$ of $M$, a (non-necessarily unique) geodesic from $x$ to $y$ of minimal length (corollary of Hopf-Rinow theorem). Recall that a subset $V$ of $M$ is geodesically convex set if for distinct points $x$ and $y$ in $V$ there exists a unique minimal geodesic segment connecting $x$ and $y$, and that unique segment is contained in $V$. For every $x$ in $M$ and $\epsilon > 0$, there exists an open subset $V$ of $M$ which is geodesically convex, has diameter less than $\epsilon$ and contains $x$ (the diameter is the supremum of the lengths of any minimal geodesic segment in $V$). 

\quad Let $U$ be an open subset of $M$. Let us say that $U$ is $k$-\emph{good} if it has not more than $k$ path components, if there exists $\epsilon > 0$ such that each path component of $U$ is geodesically convex and of diameter less than $\epsilon$, and the (geodesic) distance between any two points in distinct path components is at least $100 \epsilon$, say. The collection of all $k$-good subsets of $M$ forms a good $k$-cover of $M$. This follows from the next lemma.
\end{proof}

\begin{lem}
Suppose open subsets $U, V \subset M$ are $k$-good. Then $U \cap V$ is also $k$-good.
\end{lem} 
\begin{proof}
Choose $\epsilon_1$ which works for $U$ and $\epsilon_2$ which works for $V$. Without loss of generality, $\epsilon_1$ is less than or equal to $\epsilon_2$. Since the intersection of two geodesically convex open subsets of $M$ is a geodesically convex open subset of $M$, the components of $U \cap V$ are open, geodesically convex and of diameter less than $\epsilon_1$. 

\quad To see that there are at most $k$ components, we show that the map from $\pi_0(U \cap V)$ to $\pi_0(U)$ induced by the inclusion is injective. Suppose not. Then there exist two distinct path components of $V$ which make a nonempty intersection with a single path component of $U$. It follows that there are points $x$, $y$ in those two distinct path components of $V$ whose geodesic distance is less than $\epsilon_1$, and therefore also less than $\epsilon_2$. This contradicts our assumptions on $V$. The above argument also shows that the distance between any two points $x$, $y$ in distinct components of $U \cap V$ is at least $100 \epsilon_1$. Therefore $\epsilon_1$ works for $U \cap V$.
\end{proof}

%%%%%%%%%%%%%%%%%%%%%%%%%%%%%%%%%%%%
\begin{defn}[Coverage $\sJ_k^{\circ}$] The coverings of $M$ in $\sJ_k^{\circ}$ are the good $k$-covers of $M$.
\end{defn}

\begin{rem}
The coverage $\sJ_k$ satisfies the required axioms to be called a Grothendieck topology. The coverages $\sJ_k^{h}$ and $\sJ_k^{\circ}$ do not, and the Grothendieck topologies which they generate are too rigid to be interesting. Instead, we show in section \ref{section3} and \ref{section4} that, as \emph{homotopy} (or $\infty$) Grothendieck topologies, these coverages do generate $\sJ_k$. We approach this by proving that the three coverages define the same homotopy sheaves. We shall not need any particular prerequisites on homotopy topos theory, but we would like to suggest the reader interested in that connection to consult the works of To{\"e}n and Vezzosi \cite{toen}, Rezk, Simpson and Lurie \cite{HTT}.
\end{rem}

\section{Homotopy sheafification}\label{section1b}

\subsubsection{Projective model structure} The category $\PSh(\sC)$ of presheaves on $\sC$ has a topological or simplicial model structure, the so-called projective model structure, where weak equivalences and fibrations are determined objectwise\footnote{Meaning that a map of presheaves $F \rightarrow G$ is said to be an objectwise equivalence (resp. objectwise fibration) if the maps $F(M) \rightarrow G(M)$ are weak equivalences (resp. fibrations) in $\sS$, for each $M \in \sC$.} and cofibrations by a right lifting property with respect to acyclic fibrations. With this structure, 
\begin{enumerate}
\item every presheaf is fibrant (since every object in $\sS$ is fibrant).
\item every representable presheaf is cofibrant. This follows from the enriched Yoneda Lemma, which states that the natural map
\[
\Hom_{\PSh(\sC)}(\sC(-,X), F) \xrightarrow{\cong} F(X)
\] 
is a homeomorphism.
%\item a homotopy colimit of representable presheaves is cofibrant (Theorem 18.5.2, \cite{HH}).
\end{enumerate}

\begin{defn} The derived morphism space is the right derived functor of $\Hom$, %i.e.
\[
\mathbb{R}\Hom_{\PSh(\sC)}(X,Y) = \Hom_{\PSh(\sC)}(QX ,Y) \in \sS
\]
where $Q$ denotes a cofibrant replacement functor on $\PSh(\sC)$ with the projective model structure. \end{defn}

\begin{rem}
The usual caveat applies here: if $\sS$ is chosen to be the category of simplicial sets, then we do need to take an objectwise fibrant replacement of $Y$. Since we are working with topological spaces in mind (and every space is fibrant) this is not needed here. See the appendix for further details.
\end{rem}

\subsubsection{Local model structure} 

 Homotopy $\tau$-sheaves are the `local' objects with respect to the maps of presheaves
\begin{equation}\label{cones}
\hocolimsub{S \subset I} \sC(-,U_S) \rightarrow \sC(-,M)
\end{equation}
for each covering $\sU := \{U_i \rightarrow M\}_{i \in I}$ in $\tau$. More precisely, 

\begin{prop} 
Homotopy $\tau$-sheaves are the presheaves $F$ for which the map
\[
\RR \Hom_{\PSh(\sC)}(\sC(-,M),F) \overset{\simeq}{\longrightarrow} \RR \Hom_{\PSh(\sC)}(\hocolimsub{S \subset I} \sC(-,U_S), F)
\]
is a weak equivalence for each covering $\sU := \{U_i \rightarrow M\}_{i \in I}$ in $\tau$.
\end{prop}
\begin{proof}
The homotopy colimit on the right hand side is cofibrant (by \cite{HH}, Theorem 18.5.2) so we can consider the honest (i.e. non-derived) space of morphisms functor $\Hom_{\PSh(\sC)}$ instead. Moreover,
\[
\Hom_{\PSh(\sC)}(\hocolimsub{S \subset I} \sC(-,U_S), F) \simeq \holimsub{S \subset I} \Hom_{\PSh(\sC)}(\sC(-,U_S), F)
\]
which one can check by using the usual formulas computing hocolim/holim and cartesian closedness. The assertion now follows by applying the enriched Yoneda Lemma to both sides.
\end{proof}
\begin{defn} A morphism $F \rightarrow G$ in $\PSh(\sC)$ is a $\tau$\emph{-local equivalence} if
\[
\mathbb{R}\Hom_{\PSh(\sC)}(G, Z) \overset{\simeq}{\longrightarrow} \mathbb{R}\Hom_{\PSh(\sC)}(F, Z)
\] 
is a weak equivalence for every homotopy $\tau$-sheaf $Z$.
\end{defn}
\quad Note that if $F \rightarrow G$ is an objectwise equivalence, then it is a $\tau$-local equivalence.
\begin{rem} One can say that $\tau$-local equivalences are the maps which are seen as a weak equivalence by every homotopy $\tau$-sheaf. In this terminology, homotopy sheaves are precisely the presheaves that see all maps (\ref{cones}) as weak equivalences. 
\end{rem}

\begin{thm} There is a model structure on the underlying category $\PSh(\sC)$, called the $\tau$-local model structure and denoted by $\PSh^{\tau}(\sC)$, in which
\begin{enumerate}
\item the weak equivalences are the $\tau$-local equivalences,
\item the cofibrations are the same as in the projective model structure on $\PSh(\sC)$,
\item the fibrant objects are the homotopy $\tau$-sheaves.
\end{enumerate}
Moreover, the identity maps
\begin{equation}\label{adj}
id: \PSh(\sC) \leftrightarrows \PSh^{\tau}(\sC) : id
\end{equation}
form a simplicial/topological Quillen adjunction.
\end{thm}

\begin{proof}
This model structure is the (left) Bousfield localisation of the projective model structure on $\PSh(\sC)$ at the set of all maps of the form (\ref{cones}). The statement is then a consequence of the general theory of Bousfield localisations (for more details, see \cite{HH}).
\end{proof}

\begin{rem} In classical topos theory, further conditions are usually imposed on the allowable coverages in order to guarantee that the forgetful functor from sheaves to presheaves has a left adjoint (called sheafification) which preserves finite limits. These conditions are guaranteed  by the structure of a Grothendieck topology on $\tau$. This point of view extends naturally to the simplicial (or $\infty$) setting: if $(\sC, \tau)$ is a site (or more generally, a homotopy site), then the homotopy left adjoint in (\ref{adj}) commutes with finite homotopy limits, i.e. it is homotopy left exact (for details, consult \cite{toen}).
\end{rem}

\quad The image of a presheaf $F$ by the homotopy right adjoint $\RR id$ is the \emph{homotopy sheafification} of $F$. Equivalently, homotopy sheafification is a fibrant replacement in $\PSh^{\tau}(\sC)$, i.e. it consists of a homotopy $\tau$-sheaf $F^{sh}$ together with a $\tau$-local equivalence $F \rightarrow F^{sh}$. Two homotopy sheafifications are necessarily weakly equivalent by uniqueness (up to weak equivalence) of fibrant replacements. 

\quad We will construct an explicit homotopy $\tau$-sheafification functor in section \ref{section3} when $\sC$ is the category of $d$-manifolds and $\tau$ is $\sJ_k$.

\subsection{Taylor tower}\label{TT}

We return to the category of $d$-manifolds $\Man$. Recall from section \ref{section1} that $\Man$ has topologies $\sJ_k$, one for each non-negative $k$. By definition, every covering in $\sJ_{k+1}$ is a covering in $\sJ_{k}$ so,  given a presheaf $F$ in $\PSh(\Man)$, we obtain a tower of sheafifications

\begin{equation*}
\begin{tikzpicture}[descr/.style={fill=white}] 
\matrix(m)[matrix of math nodes, row sep=3em, column sep=3em, 
text height=1.5ex, text depth=0.25ex] 
{
F & & & & \\
F^{(0)} & F^{(1)} & F^{(2)} & F^{(3)} & ... \\}; 
\path[->,font=\scriptsize] 
	(m-1-1) edge node [auto] {} (m-2-1)
	(m-1-1) edge node [auto] {} (m-2-2) 
	(m-1-1) edge node [auto] {} (m-2-3) 
	(m-1-1) edge node [auto] {} (m-2-4)
	
	(m-2-2) edge node [auto] {} (m-2-1) 
	(m-2-3) edge node [auto] {} (m-2-2) 
	(m-2-4) edge node [auto] {} (m-2-3) 
	(m-2-5) edge node [auto] {} (m-2-4); 
	
\path[->,font=\small] 	
	(m-1-1) edge node [auto] {} (m-2-5);
\end{tikzpicture} 
\end{equation*}
\quad More precisely,
\begin{enumerate}
\item the map $F \rightarrow F^{(k)}$ is a homotopy $\sJ_k$-sheafification of $F$,
\item the map $F^{(k+1)} \rightarrow F^{(k)}$ is a homotopy $\sJ_k$-sheafification of $F^{(k+1)}$.
\end{enumerate}

\quad This tower is called the \emph{Taylor tower}. The existence of such a tower is guaranteed by the existence of Bousfield localisations in our setting. In section \ref{section3} we shall give an explicit model for the Taylor tower. It is clear that any two models are weakly equivalent by uniqueness of fibrant replacements.

%%%%%%%%%%%%%%%%%%%%%%%%%%%%%%%%%%%%%%%% %%%%%% %%%%
%%%%%%%%%%%%%%%%%%%%%%%%%%%%%%%%%%%%%%%% SECTION 2 %%%%
%%%%%%%%%%%%%%%%%%%%%%%%%%%%%%%%%%%%%%%% %%%%%% %%%%

\section{Enriched Kan extensions}\label{section2}

\quad Recall the category $\D_k$ whose objects are given by $k$ or fewer open balls (definition \ref{Ek}) and let $i$ be the inclusion $\D_k \hookrightarrow \Man$. 

\quad Let $\PSh(\D_k)$ denote the category of presheaves on $\D_k$. The restriction map $i^*: F \mapsto F \circ i$ fits into an $\sS$-enriched adjunction
\begin{equation}\label{adjunction_Kan}
i^*: \PSh(\Man) \leftrightarrows \PSh(\D_k) : Ran_i 
\end{equation}
where the right adjoint is given by $Ran_i$, the terminal or right Kan extension along $i$, as we will see below. It can be calculated as a weighted end (for details see \cite{dubuc}, Theorem I.4.2)
\[
(Ran_i G) (M) = \int_{U \in \D_k} \Hom_{\sS} (\emb(U,M), G(U))
\]
for a $G \in \PSh(\D_k)$. 

In other words, it is the equaliser of
\begin{equation*}\label{enriched_Kan}
\prod_{U \in \D_k} \Hom_\sS( \emb(U,M), G(U) ) \rightrightarrows \prod_{U, V \in \D_k} \Hom_\sS (  \emb(U, V) \times \emb(V, M),G(U) )
\end{equation*}
when evaluated at $M \in \Man$.

\begin{prop} The enriched terminal Kan extension  $Ran_i G$ of $G$ along $i$ is $\sS$-naturally isomorphic to the presheaf which assigns to a manifold $M$ the space of natural transformations from $\emb(-,M)$ to $G$, i.e.
$
\Hom_{\PSh(\D_k)}(\emb(-,M), G)
$

\end{prop}
\begin{proof}
By direct checking, using the fact that $\sS$ is cartesian closed.
\end{proof}

\quad The Yoneda lemma provides a natural transformation
\[
\epsilon : i^*Ran_i G \rightarrow G
\]
and hence a map of spaces
\begin{equation}\label{Kanext}
\Hom_{\PSh(\Man)} (Z, Ran_i G) \rightarrow \Hom_{\PSh(\D_k)} (i^*Z, i^*G)
\end{equation}
natural in $G \in \PSh(\Man)$ and $F \in \PSh(\D_k)$, obtained by applying $i$ and then post-composing with $\epsilon$. Saying that this map is a natural homeomorphism is equivalent to saying that $Ran_i$ is the right adjoint to  $i^{*}$.  One can then check this by reduction to the case of representables. Indeed, for $Z = \emb(-,M)$, the map $(\ref{Kanext})$ is a homeomorphism by the Yoneda lemma. Given an arbitrary presheaf $Z$, write it as a colimit of representables and then use the fact that $\Hom(\colim\, Z_i, G) \simeq \lim \Hom(Z_i, G)$.

\quad Since $i$ is a full embedding, $\epsilon$ is a natural homeomorphism, i.e. $\epsilon_V: Ran_i G(V) \cong G(V)$ for every $V \in \D_k$. To sum up, given be a presheaf $F$ in $\PSh(\Man)$, $Ran_i (F \circ i)$ is the best terminal approximation to $F$ by a presheaf which agrees with $F$ on $\D_k$.

\subsection{Homotopical version}
For homotopy-theoretic purposes $Ran_i$ is not appropriate, however. We need to consider the homotopical (or $\infty$) counterpart of the adjunction (\ref{adjunction_Kan}), a simplicial/topological Quillen adjunction
\begin{equation}\label{wadjunction_Kan}
i^*: \PSh(\Man) \leftrightarrows \PSh(\D_k) : \sT_k 
\end{equation}
where $\sT_k$ denotes the homotopy right adjoint to $i^*$. We proceed like in the non-homotopical case by defining a candidate for $\sT_k$ and showing it is indeed a homotopy right adjoint.

\begin{defn}\label{TkF}
Let $F \in \PSh(\D_k)$. Define the presheaf $\sT_k F$ in $\PSh(\Man)$ as
\[
(\sT_k F)(M) = \mathbb{R}\Hom_{\PSh(\D_k)} (\emb(-,M), F)
\]
\end{defn}

\begin{notn}
Note that we are restricting $\emb(-,M)$  to the full subcategory $\D_k$. So, strictly speaking, $\sT_k F(M) := \mathbb{R}\Hom_{\PSh(\D_k)} (i^*\emb(-,M), F)$, although we suppress this redundant information in the notation.
\end{notn}

%\begin{rem}
%The usual warning applies here: if $\sS$ is taken to be the category of simplicial sets one needs to take an objectwise fibrant replacement of $F|_{\D_k}$ as well.
%\end{rem}

\quad A few comments are in order:
\begin{enumerate}
\item If $M$ is in $\D_k$, then $\sT_k F (M) \simeq \Hom_{\PSh(D_k)}(\emb(-,M), F)$ since representables are cofibrant\footnote{Note, however, we are not claiming that there is a \textit{functorial} cofibrant replacement $Q$ which is the identity on representables.} in the projective model structure on $\PSh(\D_k)$. The enriched Yoneda lemma then gives a natural weak equivalence 
\[
\epsilon : i^*(\sT_k F) \rightarrow F
\]
Hence $\sT_kF$ agrees with (meaning, is objectwise equivalent to) $F$ on $\D_k$ .
\item The adjoint of the evaluation map
\begin{equation*}
\begin{array}{cccc}
ev :  & F(M) \times \emb(U,M) & \longrightarrow & F(U) \\
 	&(x, f) & \mapsto & F(f)(x)\\
\end{array}
\end{equation*}
gives rise to a morphism $F \rightarrow \Hom_{\PSh(\Dk)}(\emb(-,M),F)$  and, composing with the cofibrant replacement functor $Q$, to a morphism
\[
\eta : F \longrightarrow \sT_k F \label{map}
\]
called the $k^{th}$ \textbf{Taylor approximation} to $F$.
\item The value of $\sT_k F$ at a manifold $M$ can be presented as the totalisation of the cosimplicial object given by
\begin{equation*}
[j] \mapsto \prod_{U_0, \dots, U_j \in \D_k} \Hom_{\sS} \biggl(  \; \Bigl( \, \prod_{i = 0}^{j-1} \emb(U_{i}, U_{i+1}) \Bigr) \times \emb(U_j, M) \, , \, F(U_0) \, \biggr)
\end{equation*}

We defer the proof of this fact to the appendix.
\end{enumerate}

\quad Notice $(1)$ and $(2)$ are the counit and unit, respectively, of the hypothetical adjunction which we now show exists.

\begin{prop}\label{Tk_univ} Let $F \in \PSh(\D_k)$. Then
\begin{equation}
\mathbb{R}\Hom_{\PSh(\Man)}(G,\sT_k F) 
\xrightarrow{\simeq} \mathbb{R}\Hom_{\PSh(\D_k)}(i^*G, F)
\end{equation}
natural in $G \in \PSh(\Man)$.
\end{prop}

\begin{proof} The case of a representable functor $G = \emb(-,M)$ is straightforward from the Yoneda lemma. Given a presheaf $G$, we can resolve it by representables as in the appendix and remark (3) above. Namely, $G \simeq |\sL(G)_{\bullet}|$ where $\sL(G)_{i}$ is essentially a coproduct of representables. By bringing geometric realisation outside the $\Hom$ as a totalisation on both sides, we reduce the problem to the case of representables thus proving the claim.
\end{proof}

\begin{rem} Moreover, $\sT_k F$ is an $\infty$-full embedding in the sense that
\[
\RR \Hom_{\PSh(\Dk)}(Z,F) \rightarrow \RR \Hom_{\PSh(\Man)}(\sT_k Z, \sT_k F)
\]
is a natural weak equivalence. This follows automatically from (1) above, i.e. that the counit is a weak equivalence.
\end{rem}

\quad The proposition and remark above justify describing $\sT_k F$ as the best homotopical terminal approximation of $F$ by a functor on $\Man$ which agrees with $F$ on $\Dk$.
In what follows, we will mostly consider presheaves $F$ in $\PSh(\Man)$ so we often write $\sT_k F$ to mean $\sT_k (F \circ i)$ when no confusion should arise.

%%%%%%%%%%%%%%%%%%%%%%%%%%%%%%%%%%%%%%%% %%%%%% %%%%
%%%%%%%%%%%%%%%%%%%%%%%%%%%%%%%%%%%%%%%% SECTION 3 %%%%
%%%%%%%%%%%%%%%%%%%%%%%%%%%%%%%%%%%%%%%% %%%%%% %%%%

% !TEX root = MCpaper1.tex

\section{A model for the Taylor tower}\label{section3}
\quad This section is the heart of the paper. We show that the Taylor approximation $\sT_k F$ is a model for the homotopy sheafification of $F$ with respect to $\sJ_k$. This identifies the Taylor tower %of a presheaf $F$ with
with the tower of homotopical approximations with respect to the subcategories $\Dk$ of $\Man$.

\begin{thm}\label{Tksheaf}
The presheaf $\sT_k F$ is a homotopy $\sJ_k$-sheaf.
\end{thm}

\begin{proof}
Let $\{U_i \rightarrow M\}_{i \in I}$ be a  $\sJ_k$-cover of $M$. For each $V$ in $\Dk$, the spaces $\emb(V,M)$ and $\emb(V, U_S)$ are homotopy equivalent to the spaces of (ordered) framed configurations of $j$ points in $M$ and $U_S$ respectively, where $j$ is the number of components of $V$. The homotopy equivalence is obtained by taking the value and first derivative of the embedding at the origin of each component.
Hence, the canonical map of presheaves on $\Dk$
\begin{equation}\label{cof}
\hocolimsub{S \subset I} \emb(-,U_S) \longrightarrow \emb(-,M)
\end{equation}
is an objectwise equivalence. It follows that

\begin{equation*}
\begin{array}{ccl}
\sT_k F(M) & \overset{def}{=} & \mathbb{R}\Hom_{\PSh(\Dk)}(\emb(-,M), F) \\
		 & \simeq & \mathbb{R}\Hom_{\PSh(\Dk)}(\hocolimsub{S \subset I} \emb(-,U_S), F)\\
		 & \simeq & \holimsub{S \subset I}  \mathbb{R} \Hom_{\PSh(\Dk)}(\emb(-,U_S), F)\\
		 & = & \holimsub{S \subset I} \sT_kF(U_S)\\
\end{array}
\end{equation*}
\quad The first equivalence holds since the derived Hom preserves weak equivalences by definition. The second equivalence follows from Theorem 19.4.4, \cite{HH}.
\end{proof}

\begin{thm} \label{sheafequalsTk}
The following are equivalent for a presheaf $F \in \PSh(\Man)$.
\begin{enumerate}
\item $F$ is a homotopy $\sJ_k$-sheaf 
\item $F$ is a homotopy $\sJ_k^{\circ}$-sheaf
\item The $k^{th}$ Taylor approximation of $F$
\[
\eta_M: F(M) \overset{\simeq}{\longrightarrow} \sT_k F(M)
\]
is a weak equivalence for each $M \in \Man$.
\end{enumerate}

\end{thm}

\begin{proof}%[Proof of \ref{sheafequalsTk}]
$(1) \Rightarrow (2)$ is clear since a good $k$-cover is a $k$-cover. For $(2) \Rightarrow (3)$ take a good $k$-cover $\{U_i \rightarrow M\}_{i \in I}$ of $M$ and let $F$ be a homotopy $\sJ_k^{\circ}$-sheaf. We have the following commutative diagram
\begin{equation*}
\begin{tikzpicture}[descr/.style={fill=white}] 
\matrix(m)[matrix of math nodes, row sep=2.5em, column sep=2.5em, 
text height=1.5ex, text depth=0.25ex] 
{
F(M)  & \holimsub{S \subset I} F(U_S) \\
\sT_k F(M) &  \holimsub{S \subset I} \sT_k F(U_S) \\
}; 
\path[->,font=\scriptsize] 
	(m-1-1) edge node [auto] {$\simeq$} (m-1-2);
\path[->,font=\scriptsize]
	(m-1-1) edge node [left] {} (m-2-1);
\path[->, font=\scriptsize]	
	(m-1-2) edge node [auto] {$\simeq$} (m-2-2)
	(m-2-1) edge node [auto] {$\simeq$} (m-2-2);
\end{tikzpicture} 
\end{equation*}
where the bottom arrow is a weak equivalence by Theorem \ref{Tksheaf} and, by hypothesis, so is the top arrow. The right hand arrow is an equivalence since $F$ and $\sT_k F$ agree on $\Dk$ and $U_S \in \Dk$ by definition of a good $k$-cover.

\quad Finally, $(3) \Rightarrow (1)$ is immediate from Theorem $\ref{Tksheaf}$.
\end{proof}

\begin{thm}\label{sheafification}
The $k^{th}$ Taylor approximation of a presheaf $F$
\[
\eta: F \longrightarrow \sT_k F
\]
is a homotopy $\sJ_k$-sheafification.
\end{thm}
\begin{proof}
In theorem \ref{Tksheaf} we established that $\sT_k F$ is a homotopy $\sJ_k$-sheaf. We now show that the Taylor approximation is a $\sJ_k$-local equivalence. 

\quad Let $Z$ be a homotopy $\sJ_k$-sheaf. By theorem \ref{sheafequalsTk} the Taylor approximation of $Z$ is an objectwise equivalence, so we are required to show
\[
\mathbb{R}\Hom_{\PSh(\Man)}(\sT_k F, \sT_kZ) \longrightarrow \mathbb{R}\Hom_{\PSh(\Man)}(F, \sT_kZ)
\]
is a weak equivalence. By (\ref{Tk_univ}), the source and target of this map are weakly equivalent to $\mathbb{R}\Hom_{\PSh(\Dk)}(i^* F, i^* Z)$.
\end{proof}

%\subsection{A couple of straightforward corollaries}

\begin{cor} Let $\phi: F \rightarrow G$ be a map of homotopy $\sJ_k$-sheaves such that $i^* \phi$ is an objectwise equivalence. Then $\phi$ is an objectwise equivalence in $\PSh(\Man)$.
\end{cor}
\begin{proof}
The statement follows from the commutative diagram below.
\[
	\begin{tikzpicture}[descr/.style={fill=white}] 
	\matrix(m)[matrix of math nodes, row sep=2.5em, column sep=2.5em, 
	text height=1.5ex, text depth=0.25ex] 
	{
	F & G\\
	\sT_k F & \sT_k G \\
	}; 
	\path[->,font=\scriptsize] 
		(m-1-1) edge node [auto] {$\phi$} (m-1-2)
		(m-1-1) edge node [left] {$\simeq$} (m-2-1)
		(m-2-1) edge node [auto] {$\simeq$} (m-2-2) 
		(m-1-2) edge node [auto] {$\simeq$} (m-2-2);
	\end{tikzpicture} 
\]
The vertical arrows are weak equivalences by Theorem \ref{sheafequalsTk}. The bottom arrow is a weak equivalence by the universal property of Kan extensions (or by direct checking using the formula defining $\sT_k$).
\end{proof}

\subsection{$\sT_k$-local structure}
The homotopy idempotent functor $\sT_k : \PSh(\Man) \rightarrow \PSh(\Man)$ defines yet another model structure on $\PSh(\Man)$ by the Bousfield-Friedlander localisation of the projective model structure (see Section 9 in \cite{Bousfield}, in particular Theorem 9.3). For this new model structure, which we refer to as the $\sT_k$-local model structure, a morphism $Q : F \rightarrow G$ is
\begin{enumerate}
\item[(i)] a weak equivalence if the map \[\sT_k Q  : \sT_k F \rightarrow \sT_k G\] is an objectwise equivalence in $\PSh(\Man)$.
\item[(ii)] a fibration if it is an objectwise fibration in $\PSh(\Man)$ and the diagram
\begin{equation*}
\begin{tikzpicture}[descr/.style={fill=white}] 
\matrix(m)[matrix of math nodes, row sep=2.5em, column sep=2.5em, 
text height=1.5ex, text depth=0.25ex] 
{
F & G\\
\sT_k F & \sT_k G\\
}; 
\path[->,font=\scriptsize] 
	(m-1-1) edge node [auto] {$Q$} (m-1-2)
	(m-1-1) edge node [left] {$\eta$} (m-2-1)
	(m-1-2) edge node [auto] {$\eta$} (m-2-2)
	(m-2-1) edge node [auto] {$\sT_k Q$} (m-2-2);
\end{tikzpicture} 
\end{equation*}
is a homotopy pullback square.
\end{enumerate}

\begin{lem} Suppose $Q : F \rightarrow G$ is a morphism in $\PSh(\Man)$. Then $Q$ is $\sT_k$-local equivalence if and only if it is a $\sJ_k$-local equivalence.
\end{lem}
\begin{proof}
We show that both statements are equivalent to the assertion
\begin{equation}\tag{*}
\mbox{the restriction $i^* Q$ is an objectwise weak equivalence on $\PSh(\Dk)$}
\end{equation}

If $Q$ is a $\sT_k$-local equivalence, this is immediate from the definition of $\sT_k$. Suppose now that $Q$ is a $\sJ_k$-local equivalence. Using Theorem \ref{sheafequalsTk} which identifies a homotopy sheaf $Z$ with $\sT_k Z$, we have that the induced map
\begin{equation*}
\mathbb{R}\Hom_{\PSh(\Man)}(G, \sT_k Z) \rightarrow \mathbb{R}\Hom_{\PSh(\Man)}(F, \sT_k Z)
\end{equation*}
is a weak equivalence for every homotopy $\sJ_k$-sheaf $Z$ in $\PSh(\Man)$. By adjunction, this is equivalent to 
\begin{equation}\label{AA}
\mathbb{R}\Hom_{\PSh(\Dk)}(i^* G, i^*Z) \rightarrow \mathbb{R}\Hom_{\PSh(\Dk)}(i^*F, i^*Z)
\end{equation}
being a weak equivalence for every homotopy sheaf $Z$. Since homotopy $\sJ_k$-sheaves are determined by their value on $\Dk$, we now see this is equivalent to $(*)$ rephrased as saying that the natural map
\begin{equation}\label{AA}
\mathbb{R}\Hom_{\PSh(\Dk)}(i^* G, W) \rightarrow \mathbb{R}\Hom_{\PSh(\Dk)}(i^*F, W)
\end{equation}
is a weak equivalence for every $W \in \PSh(\Dk)$.
\end{proof}

\quad The two model structures have the same cofibrations by definition and the same weak equivalences by the preceding lemma, so the fibrations coincide.
\begin{cor}\label{qe}
The $\sT_k$-local and $\sJ_k$-local model structures on $\PSh(\Man)$ coincide. In particular, the identity functors yield a Quillen equivalence.
\end{cor}
\quad  It is worth emphasising that via the $\sT_k$-local structure we have completely described the fibrations in the $\sJ_k$-local structure, something which a priori was not known.

%%%%%%%%%%%%%%%%%%%%%%%%%%%%%%%%%%%%%%%% %%%%%% %%%%
%%%%%%%%%%%%%%%%%%%%%%%%%%%%%%%%%%%%%%%% SECTION 4 %%%%
%%%%%%%%%%%%%%%%%%%%%%%%%%%%%%%%%%%%%%%% %%%%%% %%%%

\section{Connection to operads}

\quad For each positive $k$, fix an embedding $\eta_k$ of the disjoint union of $k$ copies of $\RR^d$ in $\RR^{\infty}$. By taking the images of the embeddings $\eta_k$ we obtain a category which is a topological skeleton of $\D_{\infty}$. This category (which we will still refer to as $\D_{\infty}$) is a topological PROP, i.e. its objects are identified with the non-negative integers, and it has a symmetric monoidal structure (here given by disjoint union) which corresponds to the addition of integers. $\D_{\infty}$ is called \textit{framed little $d$-discs} PROP. The framed little $d$-discs operad is the part $\Dn{\infty}(m,1) := \emb(\amalg_m \RR^d, \RR^d)$ of the PROP and since
\[\emb(m,n) \cong \coprod_{f : \underline{m} \rightarrow \underline{n}} \emb(m_1,1) \times \dots \times \emb(m_n,1)\]
where $m_i$ denotes the preimage of $i \in \underline{m} = \{1, \dots, m\}$ by $f$, we can reconstruct the PROP from the operad and vice-versa and use the two words interchangeably.

\quad Moreover, the category $\PSh(\D_{\infty})$ is $\sS$-isomorphic to the category of right modules over the framed little discs operad $P$, denoted $\Mod_P$. Therefore, for a given $F \in \PSh(\Man)$, we obtain a description of $\sT_{\infty} F$ as a derived space of right module maps over the framed little discs operad,
\begin{equation}\label{operads}
\sT_{\infty} F (M) \simeq \mathbb{R}\Hom_{P}( \emb_{M}, F)
\end{equation}
where the two obvious right $P$-modules are $\emb_M (n) := \emb(\amalg_{n} \RR^d, M)$ and $F(n) := F(\amalg_n \RR^d)$. This answers a conjecture (4.14,  \cite{AroneTurchin}) of G. Arone and V. Turchin.

\quad Combining (\ref{operads}) with the analyticity results of Goodwillie-Klein for the embedding functor one has the following immediate consequence.
\begin{prop} Suppose $\textup{dim}(N) - \textup{dim}(M) \geq 3$. Then 
\[
\emb(M,N) \simeq \mathbb{R}\Hom_{P}( \emb_{M}, \emb_{N})
\]
\end{prop}

\begin{rem}
For finite $k$, we obtain `truncated' versions of (\ref{operads}). The category $\PSh(\Dk)$ is $\sS$-isomorphic to the category of $k$-truncated right modules over the $k$-truncated framed little discs operad. The composition product on the the category of $k$-truncated (symmetric) sequences is the obvious one,
\[M(n) \times M(m_1) \times \dots \times M(m_n) \rightarrow M(m_1 + \dots + m_n)\]
only defined when $m_1 + \dots + m_n \leq k$. 

\quad Specialising to $n \leq k$, we view $\emb_M$ and $F$ as $k$-truncated sequences of spaces.  In particular, $\emb_M$ and $F$ are $k$-truncated modules over the $k$-truncated framed little discs operad $P_k := \{P(n)\}_{n \leq k}$, and we see that
\begin{equation}\label{tr_operads}
\sT_{k} F (M) \simeq \mathbb{R}\Hom_{P_k}( \emb_{M}, F)
\end{equation}
\end{rem}

\vspace{0.2cm}
\quad Another example of interest is the singular chains of the embedding functor, $S_{*}\emb(-,N)$. We will briefly sketch how to obtain a chain complex version of $\sT_k$. We write $S_*$ for the normalised singular chains functor $\mathsf{Top} \rightarrow \Ch$. Since it is a lax monoidal functor, we can use it to enrich $\Man$ over chain complexes.

\quad Rename, for the rest of this section, $\PSh(\Man)$ (resp. $\PSh(\Dk)$) as the category of $\Ch$-enriched presheaves from $\Man$ (resp. $\Dk$) to $\Ch$ and define
\[
\sT_k^{\mathsf{Ch}} F (M) := \RR \Hom_{\PSh(\Dk)}(S_*\emb(-,M), F) \in \Ch \]

\quad The arguments of the previous sections show that $F \rightarrow \sT_{k}^{\mathsf{Ch}}F$ is a homotopy $\sJ_k$-sheafification. It is also not hard to show that $\sF_k$ is $\Ch$-equivalent to the category of right modules over $S_*P$, the chains of the framed little discs operad. Hence,
\[
\sT_{\infty}^{\mathsf{Ch}} F (M) \simeq \RR \Hom_{S_*P}(S_*\emb_M, F)
\]
\quad In the particular case when $F$ is the (normalised) singular chains of $\emb(-,N)$ we obtain the following result, by the analyticity results in \cite{homology}.
\begin{prop} Suppose $2 \, \textup{dim}(M) + 1 < \textup{dim}(N)$. The Taylor approximation gives a chain homotopy equivalence
\[
S_*\emb(M,N) \simeq \RR \Hom_{S_*P}(S_*\emb_M, S_*\emb_N)
\]
natural in $M$ and $N$.
\end{prop}

%%%%%%%%%%%%%%%%%%%%%%%%%%%%%%%%%%%%%%%% %%%%%% %%%%
%%%%%%%%%%%%%%%%%%%%%%%%%%%%%%%%%%%%%%%% SECTION 5 %%%%
%%%%%%%%%%%%%%%%%%%%%%%%%%%%%%%%%%%%%%%% %%%%%% %%%%

\section{Homotopy $\sJ_k$-sheaf = Polynomial functor}\label{section4}

\quad Recall from section \ref{section1} that a polynomial functor of degree $\leq k$ is a homotopy sheaf for the coverage $\sJ_k^h$.

\begin{defn}
A presheaf $F \in \PSh(\Man)$ is \textbf{good} if, for any sequence $U_0 \subset U_1 \subset \dots$ in $\Man$ whose union is $M$, the natural map\
\[
F(M) \longrightarrow \holimsub{i} F(U_i)
\]
is a weak equivalence of spaces.
\end{defn}

\begin{thm}\label{polequalsJk}
The following are equivalent
\begin{enumerate}
\item $F$ is a homotopy $\sJ_k$-sheaf
\item $F$ is good and polynomial of degree $\leq k$. 
\end{enumerate}
\end{thm}

\begin{proof}
A covering in $\sJ_k^{h}$ is a covering in $\sJ_k$ so in order to show $(1) \Rightarrow (2)$ we need only prove goodness. Observe that a covering $\{U_i \rightarrow M\}_{i \in \NN}$ of $M$ with $U_i \subset U_{i+1}$ is a $k$-cover and
\[
\holimsub{\varnothing \neq S \subset I} F(U_S) \simeq  \holimsub{i} F(U_i)
\]
so the homotopy $\sJ_k$-sheaf property for these coverings is precisely the condition of goodness.

\quad Now, suppose $F \in \PSh(\Man)$ is good and polynomial of degree $\leq k$. By Theorem \ref{sheafequalsTk}, we are required to show that 
\begin{equation}\label{sss}
F(M) \rightarrow \sT_kF(M)
\end{equation}
 is a weak equivalence for every $M \in \Man$. The proof is now essentially the same as the one of Theorem 5.1 in \cite{embI}.
 
\quad Due to the goodness of $F$, we can assume that $M$ is the interior of a compact handlebody $L$. Take a handle decomposition of $L$ with top-dimensional handles of index $s$.

\quad The first case is $s = 0$, i.e. $M \simeq \amalg_i \RR^d$ for some $i$. If $i \leq k$, then $F(M) \simeq \sT_k F(M)$ since $F$ and $\sT_k F$ agree on $\Dk$ by construction. For $i > k$ we proceed inductively. Choose $k+1$ distinct components $A_0, \dots, A_k$ of $M$ and consider the commutative diagram
\[
	\begin{tikzpicture}[descr/.style={fill=white}] 
	\matrix(m)[matrix of math nodes, row sep=2.5em, column sep=2.5em, 
	text height=1.5ex, text depth=0.25ex] 
	{
	F(M) & \sT_k F(M)\\
	 \holimsub{\varnothing \neq S \subset \{0,\dots,k\}} F(M \backslash A_S) & \holimsub{\varnothing \neq S \subset \{0,\dots,k\}} \sT_k F(M \backslash A_S) \\
	}; 
	\path[->,font=\scriptsize] 
		(m-1-1) edge node [auto] {$\eta$} (m-1-2)
		(m-1-1) edge node [left] {$\simeq$} (m-2-1)
		(m-2-1) edge node [auto] {$\simeq$} (m-2-2) 
		(m-1-2) edge node [auto] {$\simeq$} (m-2-2);
	\end{tikzpicture} 
\]
\quad The vertical arrows are weak equivalences since $F$ is polynomial of degree $\leq k$ by hypothesis, and the horizontal arrow is a weak equivalence by induction.

\quad Now, suppose $s > 0$. Pick one of the $s$-handles
\[
e : D^{d-s} \times D^s \longrightarrow L
\]
where $e^{-1}(\partial L) = D^{d-s} \times S^{s-1}$. 

\quad Take pairwise disjoint closed discs $C_0, \dots, C_k$ in the interior of $D^s$ and define
\[
A_i := e(D^{d-s} \times C_i) \cap M
\]
\quad Then,
\begin{enumerate}
\item $A_i$ is closed in $M$ and $M \backslash A_i$ is the interior of a smooth handlebody with a handle decomposition with fewer s-handles, and so is any intersection of these, $\cap_{i \in S} M \backslash A_i$, for $S$ a non-empty subset of $\{0,\dots, k\}$.
\item The family $\{ M \backslash A_i \rightarrow M \}_{i \in \{0,...,k\}}$ is a covering of $M$ for the coverage $\sJ_{k}^{h}$.
\end{enumerate}

\quad By induction, as in the case $s = 0$, the statement is easily verified.
\end{proof}

\begin{rem}
A way to paraphrase the theorem above is as follows. Define a coverage by declaring $Cov(X)$ to consist of coverings in $\sJ_k^h$ and coverings of the form $\{U_i \rightarrow M\}_{i \in \NN}$ with $U_i \subset U_{i+1}$. Then Theorem \ref{polequalsJk} says that this coverage and $\sJ_k$ define the same homotopy sheaves.
\end{rem}

\quad Polynomial functors are important in manifold calculus (and in functor calculus in general) as they  can be given rather explicit  descriptions in terms of cubical diagrams. The coverings in $\sJ_k^h$ are in practice much smaller than arbitrary or good $k$-covers so they are easier to handle.

\begin{expl} 
A polynomial functor of degree $\leq 1$ is a functor $F$ which sends homotopy pushout squares to homotopy pullback squares. 
\end{expl} 
\quad
\subsection{Classification of linear functors}
Following Goodwillie, we call a presheaf $F$ \textit{reduced }Êif $F(\varnothing) \simeq *$.
Most examples of interest are reduced, but not every homotopy $\sJ_1$-sheaf is reduced: take for instance a constant sheaf. One can reduce a presheaf $F$ by setting 
\[F^{red} := \hofiber(F \rightarrow F(\emptyset))\]

\begin{defn}
A reduced polynomial functor of degree $\leq 1$ is called a \emph{linear} functor.
\end{defn}

 If $F$ is reduced, then
\begin{equation*}
\begin{array}{ccl}
\sT_1 F(M) & \simeq & \Hom^{hO(d)}(\emb(\RR^d, M), F(\RR^d)) \\
		 & \simeq & \Hom^{O(d)}(\textup{frame}(M), F(\RR^d)) \\
		 & \simeq & \Gamma(\textup{frame}(M) \times_{O(d)} F(\RR^d) \rightarrow M)\\
\end{array}
\end{equation*}
where $\textup{frame}(M)$ denotes the total space of the tangent frame bundle of $M$, $\Hom^{O(d)}(-,-)$ the space of $O(d)$-maps and $\Hom^{h O(d)}(-,-)$ its derived functor. The first equivalence above follows from the fact that $\D_1(\RR^d,\RR^d) := \emb(\RR^d, \RR^d) \simeq O(d)$ and $F(\varnothing) \simeq *$. 

\quad  Combining Theorems \ref{sheafequalsTk} and \ref{polequalsJk} with the above paragraph we obtain

\begin{prop} The following are equivalent for a presheaf $F \in \PSh(\Man)$.
\begin{enumerate}
\item $F$ is linear and good
\item The `scanning' map
 \[
F(M) \longrightarrow  \Gamma( \textup{frame}(M) \times_{O(d)} F(\RR^d) \rightarrow M)
\]
is a natural weak equivalence.
\end{enumerate}
\end{prop}

\begin{comment}
\subsection{Classification of homogeneous functors}

\begin{defn} A presheaf $H \in \PSh(\Man)$ is $k$-homogeneous if it is a homotopy $\sJ_k$-sheaf and is $k$-reduced in the sense that $H(U) \simeq *$ for every $U \in \D_{k-1}$.
\end{defn}

The prototypical example is of course
\[
L_k F := \hofibre ( \sT_k F \rightarrow \sT_{k-1} F )
\]

One can use (3) in \dots to show that a presheaf $H$ is $k$-homogeneous if and only if
\begin{equation}\label{homo}
H(M) \simeq \{ \alpha \in \Hom^{hAut(k)}(\overline{\emb}(k,M), H(k)) : \alpha |_{\partial} \simeq * \}
\end{equation}
where $Aut(k)$ denotes the space of embeddings of $\amalg_k \RR^{d}$ into itself which are bijective on $\pi_0$; $\overline{\emb}(k,M)$ is the compactified configuration space of $k$ (framed) points in $M$, $\partial$ is its boundary and it encodes the collision of points. Lastly, $H(k)$ is short for $H(\amalg_k \RR^d)$.

\quad A rephrasing of this, closer to \cite{embI}, is to consider $\emb(\amalg_k \RR^d,M)$ (or the homotopy equivalent space of framed configurations of $k$ points in $M$) as a subspace of $M^{k}$, and let $\partial$ be a small neighbourhood of the fat diagonal. Then (\ref{homo}) is weakly equivalent to
\[
\{ \alpha \in \Gamma[Ê\mathsf{Conf}(k,M) \wedge_{Aut(k)} H(k) \rightarrow \mathsf{Conf}(k,M)] : \alpha|_{\partial} = * \}
\]

\end{comment}

%%%%%%%%%%%%%%%%%%%%%%%%%%%%%%%%%%%%%%%% %%%%%% %%%%
%%%%%%%%%%%%%%%%%%%%%%%%%%%%%%%%%%%%%%%% SECTION 6 %%%%
%%%%%%%%%%%%%%%%%%%%%%%%%%%%%%%%%%%%%%%% %%%%%% %%%%

\section{Relation to the unenriched model}
\quad Let $\sO(M)$ be the poset of open sets of a manifold $M$, i.e. the discrete and relative version of $\Man$. Clearly there is an `inclusion' functor
\[
\sO(M) \rightarrow \Man
\]
given by inclusion $Ob(\sO(M)) \hookrightarrow Ob(\Man)$ on object-sets and sending a morphism $U \subset V$ in $\sO(M)$ to the inclusion $i \in \emb(U,V)$.
\begin{defn} A functor $f : \sO(M)^{op} \rightarrow \sS$ is called \textit{context-free}\footnote{This terminology is due to G. Arone and V. Turchin} if it factors through $\Man$ by an $\sS$-functor $F : \Man^{op} \rightarrow \sS$.
\end{defn}

\begin{rem}
If $f$ is context-free, then it is necessarily isotopy invariant.
\end{rem}

\quad The discrete analogues of $\Dk$ and $\sT_k$ are denoted $\sO_k$ and $T_k$ respectively. We refer the reader to \cite{embI} for details on the unenriched setting. The following proposition says that $\sT_k$ is really an enrichment of $T_k$.
\begin{prop} Let $f$ be a context-free functor on $\sO(M)$. Then
\[
T_k f (U) \simeq \sT_k F (U)
\]
for every $U \in \sO(M)$.
\end{prop}

\begin{proof}
By definition, $T_k f (U) := \holimsub{V \in \sO_k(U)} f(V)$. Then,

\begin{equation*}
\begin{array}{ccll}
T_k f(M) %& \overset{def}{=} & \holimsub{V \in \sO_k(U)} f(V) &\\
		 & = &\holimsub{V \in \sO_k(M)} F(V) & \\
		 & \cong & \holimsub{V \in \sO_k(M)} \Hom_{\PSh(\Dk)}(\Dk(-,V), F) & \\
		 & \simeq & \Hom_{\PSh(\Dk)}(\hocolimsub{V \in \sO_k(M)} \Dk(-,V), F) &\\
		 & \simeq & \mathbb{R}\Hom_{\PSh(\Dk)}(\emb(-,M), F) & \\
\end{array}
\end{equation*}
The first weak equivalence holds because $f$ is context-free. The second is the enriched Yoneda lemma. The last equivalence follows from the fact that the map of presheaves in $\PSh(\Dk)$
\begin{equation}\label{bah}
\hocolimsub{V \in \sO_k(M)} \emb(-,V) \rightarrow \emb(-,M)
\end{equation}
 is an objectwise equivalence since, again, $\emb(-,V)$ and $ \emb(-,M)$ in $\PSh(\Dk)$ are framed configuration spaces. Moreover, the left-hand side is cofibrant in the projective model structure since representables $\emb(-,V)$ are cofibrant and the homotopy colimit of an objectwise cofibrant diagram in a simplicial model category is cofibrant by Theorem 18.5.2, \cite{HH}.\end{proof}

%%%%%%%%%%%%%%%%%%%%%%%%%%%%%%%%%%%%%%%% %%%%%% %%%%
%%%%%%%%%%%%%%%%%%%%%%%%%%%%%%%%%%%%%%%% SECTION 7 %%%%
%%%%%%%%%%%%%%%%%%%%%%%%%%%%%%%%%%%%%%%% %%%%%% %%%%

\section{Boundary case}
\quad Fix a $(d-1)$-manifold $Z$, and let $\Man^{\partial}$ denote the category of $d$-manifolds $M$ with boundary $\partial M$, with a chosen diffeomorphism to $Z$. For simplicity, we assume $Z$ is connected. Morphisms are (neat) embeddings respecting the identification of boundaries with $Z$. The replacement for $\Dk$ is $\Dbk$, the full subcategory of $\Man^{\partial}$ whose objects are identified with non-negative integers (i.e. an object is the disjoint union of $Z \times [0,1)$ with $n$ copies of the disc $\RR^d$). Notice that a morphism in $\Dbk$ may take some of the discs in the source to $Z \times [0,1)$.

\quad In parallel with the non-boundary case, define $\sJ_k$ as the Grothendieck topology on $\Man^{\partial}$ with coverings given by collections $\{U_i \rightarrow M\}$ subject to the conditions that every finite subset of cardinality $k$ in the interior of $M$ is contained in $U_j$ for some $j$.

\quad A good $k$-cover $\{U_i \rightarrow M \}$ is then a $\sJ_k$-cover with the property that every finite intersection $U_{i_0} \cap \dots \cap U_{i_n}$ is in $\Dbk$, for $n \geq 0$.

\quad The boundary versions of the statements in sections $5$, $6$, $7$ and $8$ follow from the propositions below.
\begin{prop}\label{cl1}
 Every manifold $M$ with boundary admits a good $k$-cover.
\end{prop}
\begin{prop}\label{cl2}
For every $\sJ_k$-covering $\{U_i \rightarrow M\}_{i \in I}$ and every $V$ in $\Dbk$, the map
\[
\hocolimsub{\varnothing \neq S \subset I} \emb^{\partial}(V,U_S) \rightarrow \emb^{\partial}(V,M)
\]
is a weak equivalence.
\end{prop}

Here $\emb^{\partial}(V,Y)$ denotes the space of embeddings of $V$ into $M$ fixing the boundary. 
Proposition \ref{cl2} can be proved by analogy with the non-boundary case statement appearing in the proof of Theorem \ref{Tksheaf}. 

\quad A variation of the argument in the proof of Proposition \ref{admits_good} proves Proposition \ref{cl1} as we now briefly describe. Equip $M$ with a complete Riemannian metric which restricts to a product metric on a fixed collar $C$ of the boundary. We now say that an open subset $U$ is $k$-good if it is the disjoint union of a sub-collar $C^{\prime}$ of $C$ (on which the metric also restricts to a product) and $m$ components $U_1, \dots, U_m$ diffeomorphic to $\RR^d$, with $m \leq k$, subject to the conditions that each path component of $U$ is geodesically convex, there exists $\epsilon > 0$ such that the diameter of the ${U_i}$ is less than $\epsilon$, and the distance between any two points in distinct components (including the collar) of $U$ is at least $100 \epsilon$, say. The collection of all $k$-good subsets of $M$ then forms a good $k$-cover. \\

\quad We thus obtain a Taylor tower for a presheaf $F \in \PSh(\Man^{\partial})$ where
\[
\sT_k F (M) = \RR \Hom_{\PSh(\Dbk)}( \emb^{\partial}(-,M), F)
\]
is a model for the $k^{th}$-approximation of $F$.

%%%%%%%%%%%%%%%%%%%%%%%%%%%%%%%%%%%%%%%% %%%%%% %%%%
%%%%%%%%%%%%%%%%%%%%%%%%%%%%%%%%%%%%%%%% APPENDIX %%%%
%%%%%%%%%%%%%%%%%%%%%%%%%%%%%%%%%%%%%%%% %%%%%% %%%%

\appendix

\section{Derived mapping spaces and resolutions}\label{appendix}
\quad Throughout this appendix, we let $\sC$ denote either $\Man$ or $\Dk$ and endow $\PSh(\sC)$ with the projective model structure. Recall this means that weak equivalences and fibrations are determined objectwise.

\subsection{Derived mapping spaces}
The category $\PSh(\sC)$ of $\sS$-enriched functors on $\sC$ is $\sS$-enriched. Given $X$ and $Y$, we denote its enriching morphism object by
\[
\Hom_{\PSh(\sC)}(X,Y)
\]
\quad We now want to make a distinction between simplicial sets and compactly generated Hausdorff spaces (CGHS). 

\quad If $\sS$ = simplicial sets, then $\Hom_{\PSh(\sC)}(X,Y)$ is the simplicial set whose set of $n$-simplices is given by set of natural transformations $X \otimes \Delta[n] \rightarrow Y$, where $(X \otimes \Delta[n])(U) := X(U) \times \Delta[n]$. This makes $\PSh(\sC)$ into a simplicial model category.

\quad If $\sS$ = CGHS, then $\Hom_{\PSh(\sC)}(X,Y)$ denotes the space we obtain by topologising the set of natural transformations $X \rightarrow Y$ by the (kelleyfication of the) subspace topology of the product
\[
\prod_{U \in \sC} \Hom_{\sS}(X(U), Y(U))
\]
equipped with the product topology.

\quad If $\sS$ is CGHS, we can still define a simplicial model structure on $\PSh(\Man)$. The simplicial set of natural transformations between $X$ and $Y$, which we denote by $\Map(X,Y)$, has the following set of $n$-simplicies 
\[
\Map(X,Y)_n := \Hom_{\PSh(\Man)}(X \otimes \Delta[n],Y)
\]
where $(X \otimes \Delta[n])(U) := X(U) \times |\Delta[n]|$, $U \in \Man$.

\quad These two enrichments are related by
\[
\Map(X,Y) \simeq \mathsf{Sing}(\Hom_{\PSh(\Man)}(X,Y))
\]

\begin{defn} The derived mapping space functor is the right derived functor of $\Hom_{\PSh(\sC)}$,
\[
\mathbb{R}\Hom_{\PSh(\sC)}(X,Y) := \Hom_{\PSh(\sC)}(QX ,RY) \in \sS
\]
where $Q$ and $R$ denote, respectively, cofibrant and fibrant replacement functors.
\end{defn}

\quad The `derived' version of $\Map$ is the homotopy function complex of Dwyer-Kan, denoted by $\hMap$. In fact, since $\Map$ makes $\PSh(\sC)$ into a simplicial model category, a model for $\hMap$ is given (\cite{DwyerKan}, Cor. 4.7) by
$\Map(QX,Y)$
where $Q$ denotes a cofibrant replacement functor on $\PSh(\sC)$. Moreover, $\hMap(X,Y) \simeq \mathsf{Sing}(\mathbb{R}\Hom_{\PSh(\sC)}(X,Y))$.

%%%%%%%%%%%%%%%%%%%%%%%%%%

\subsection{Resolutions} In this section we discuss the construction of a resolution of a presheaf $F \in \PSh(\Dk)$ (c.f. 2.6 in \cite{Dugger}). More precisely, we wish to find a cofibrant presheaf $\widetilde{F}$ and a weak equivalence $\widetilde{F} \rightarrow F$, where everything in sight should be enriched as always.

\subsubsection{Free presheaves on $\Dk$}
Let $\Dk^\delta$ denote\footnote{The symbol $\delta$ stands for `discrete'.} the category with the same objects as $\Dk$ but only identity morphisms. Define $\PSh(\Dk^{\delta})$ to be the category of  contravariant functors from $\Dk^\delta$ to $\sS$ and consider the following free-forgetful adjunction
\begin{equation}
L: \PSh(\Dk^{\delta}) \leftrightarrows \PSh(\Dk) : U
\end{equation}
where $U$ is the obvious forgetful functor and $L$ is a left adjoint to $U$. In other words, $L(G)$ is the (enriched) left Kan extension of a presheaf $G \in \PSh(\Dk^{\delta})$ along the inclusion $i : \Dk^{\delta} \hookrightarrow \Dk$. More concretely, 
\[
 L(G) 
 =\Hom_{\Dk}(-,i) \otimes_{\Dk^\delta} G 
% = \int^{V \in \Dk^\delta} \emb(-, i(V)) \times G(V) 
= \coprod_{V \in \Dk} \emb(-,V) \times G(V)
\]
\quad Also note that the free-forgetful adjunction $(L,U)$ is enriched so, in particular, we obtain homeomorphisms
\begin{equation}\label{adj_free}
\Hom_{\PSh(\Dk)}(L(G), F) \cong \Hom_{\PSh(\Dk^{\delta})}(G, U(F)) = \prod_{V \in \Dk} \Hom_{\sS}\bigl(G(V), F(V)\bigr)
\end{equation}
which are $\sS$-natural in $F$ and $G$.

\subsubsection{Cotriple resolution} Associated to the free-forgetful adjunction we construct a simplicial object in $\PSh(\Dk)$, usually called the cotriple resolution. For $G$ in $\PSh(\Dk)$, let $\sL(G)_{\bullet}$ be the simplicial object with $n$-simplicies given by
\[
\sL(G)_{n} := (LU)^{n+1}(G) \in \PSh(\Dk)
\]
i.e.
\[
\coprod_{V_0, \dots, V_n \in \Dk} \emb(-,V_0) \times \dots \times \emb(V_{n-1}, V_{n}) \times G(V_n)
\]
\quad Note that $\sL(G)_{\bullet}$ is naturally augmented via the map $LU(G) \rightarrow G$ given by the composition (remember $G$ is contravariant)
\[
\begin{array}{ccc}
\emb(W,V) \times G(V) & \longrightarrow & G(W) \\
(f, x) & \longmapsto & G(f) (x)
\end{array}
\]
\quad Finally, we define $|G| \in \PSh(\Dk)$ to be the geometric realisation of $\sL(G)_{\bullet}$,
\[
|G| := | \sL(G)_{\bullet} |
\]

\begin{prop}\label{realisation_cof} The presheaf $|G|$ is a cofibrant replacement of $G$ in $\PSh(\Dk)$.
\end{prop}
\begin{proof}
The natural map $|G| \rightarrow G$ is a weak equivalence by general considerations of cotriple resolutions. Moreover, $|G|$ is cofibrant since $\sL(G)_{\bullet}$ is Reedy cofibrant and geometric realisation preserves cofibrations. %Alternatively, we can proceed more directly by showing that all the maps $sk_{n-1} |Z| \rightarrow sk_{n}  |Z|$ are cofibrations. Since $|Z|$ is the colimit of $ sk_{n} |Z|$, the result then follows.
\end{proof}

\quad Hence, for presheaves $F$ and $G$ in $\PSh(\Dk)$, $\RR \Hom_{\PSh(\Dk)}(G, F)$ is weakly equivalent to $\Hom_{\PSh(\Dk)}(|G|, F)$. Then, 
$
\Hom_{\PSh(\Dk)}(|G|, F) \simeq \textup{Tot} \, \Hom_{\PSh(\Dk)}(\sL(G)_{\bullet},F) 
$.
\begin{thm}\label{tot} Let $F$ be a presheaf. Then $\sT_k F(M)$ is weakly equivalent to the totalisation of the cosimplicial object whose space of $0$-simplicies is
\[
\prod_{V \in \Dk} \Hom_{\sS}(\emb(V,M), F(V))
\]
and, for $n > 0$, whose space of $n$-simplicies is
\[
\prod_{V \in \Dk} \Hom_{\sS}(\sL(M)_{n-1}(V), F(V))
\]

where $\sL(M)_{\bullet} := \sL(\emb(-,M))_{\bullet}$.
\end{thm}

\begin{proof}
By the previous proposition we know
\[
\sT_k F (M) \simeq \textup{Tot}\,  \Hom_{\PSh(\Dk)}(\sL(M)_{\bullet}, F)
\]
The result follows by applying the adjunction (\ref{adj_free}).
\end{proof}

\bibliography{biblio}
\bibliographystyle{amsalpha}

\end{document}